\documentclass[10pt]{amsart}

\input{preamble.sty}

\title{Univalent categories of modules}
\author{Jarl G. Taxerås Flaten}
\address{Department of Mathematics, University of Western Ontario, London, Ontario, Canada}
\email{jtaxers@uwo.ca}
\urladdr{https://publish.uwo.ca/~jtaxers}
\date{July 1, 2022}

\begin{document}

\maketitle

\begin{abstract}
  We show that categories of modules over a ring in Homotopy Type Theory (HoTT) satisfy the internal versions of the AB axioms from homological algebra.
  The main subtlety lies in proving AB4, which is that coproducts indexed by arbitrary sets are left-exact.
  To prove this, we replace a set with its strict category of (ordered) finite sub-multisets.
  From showing that the latter is filtered, we deduce left-exactness of the coproduct.
  More generally, we show that exactness of filtered colimits (AB5) implies AB4 for any abelian category in HoTT.
  Our approach is heavily inspired by Roswitha Harting's construction of the {\em internal coproduct} of abelian groups in an elementary topos with a natural numbers object~\cite{Har82}.

  To state the AB axioms we define and study filtered (and sifted) precategories in HoTT.
  A key result needed is that filtered colimits commute with finite limits of sets.
  This is a familiar classical result, but has not previously been checked in our setting.

  Finally, we interpret our most central results into an \(\infty\)-topos \( \XX \).
  Given a ring \( R \) in \( \XX \), we show that the internal category of \(R\)-modules in \( \XX \) represents the presheaf which sends an object \( X \in \XX \) to the category of \( (X{\times}R) \)-modules over \( X \).
  In general, our results yield a product-preserving left adjoint to base change of modules over \( X \).
  When \( X \) is \(0\)-truncated, this left adjoint is the internal coproduct.
  By an internalisation procedure, we deduce left-exactness of the internal coproduct as an ordinary functor from its internal left-exactness coming from HoTT.
\end{abstract}

\tableofcontents

\section{Introduction}
We study categories of modules over a ring in Homotopy Type Theory (HoTT).
Our main result is that these satisfy the (internal) axioms AB3 through AB5 and have a generator, i.e., they are {\em Grothendieck categories}.
By working in HoTT our results hold in any (Grothendieck) \( \infty \)-topos~\cite{Shu19}, and conjecturally in any elementary \( \infty \)-topos~\cite{KL18,Ras22,Shu17}.
In \cref{sec:semantics}, we interpret our most central results into an \( \infty \)-topos.
This work is part of, and motivated by, the development of homological algebra in HoTT.

In ordinary homological algebra, it is common knowledge that the category of modules over a ring is Grothendieck and satisfies AB4.
However, the question is more subtle in a constructive setting such as ours.
For example, the category of abelian groups in the type theory of~\cite{CS10} is only preabelian (see their Section~4.1 for a discussion).
Fortunately for us, \( \Mod{R} \) is abelian in HoTT, and this has already been formalized for \( R \jeq \Zb \) in the UniMath library~\cite{unimath}.

The main subtlety in verifying that \( \Mod{R} \) is Grothendieck is the existence of coproducts over an arbitrary set \( X \).
When assuming the law of the excluded middle, we are accustomed to having a natural monomorphism \( \bigoplus_{x:X} A(x) \to \Pi_{x:X} A(x) \) from an arbitrary coproduct of modules to the corresponding product.
Indeed, one often defines the coproduct to be the ``finitely supported'' elements within the product.
While the coproduct \( \bigoplus_{x:X} A(x) \) still always exists in a constructive setting, it is harder to define, and in contrast to the classical setting there may be no non-trivial maps (let alone monomorphisms) of the form \( \bigoplus_{x:X} A(x) \to \Pi_{x:X} A(x) \)!
This is further discussed in Section~\ref{ssec:grothendieck-categories}.

When Grothendieck first introduced the AB axioms, he remarked that AB4 follows from AB5~\cite[129]{Gro57}.
This is the second point which is a bit more subtle in our setting, and we prove this in Section~\ref{ssec:coproducts-a-la-harting}.
In fact, we prove a bit more: the AB5 axiom implies that the coproduct functor \( \bigoplus_X \) is left-exact for arbitrary sets \( X \) (Theorem~\ref{thm:coproduct-lex}).
Our result is analogous to, and inspired by, the {\em internal coproduct} of a family of abelian groups in an elementary topos (with $\Nb$) as constructed by Roswitha Harting in~\cite{Har82}.
Her main result is that the internal coproduct, indexed by an arbitrary object, is left-exact.
In~\cite{Ble18}, Ingo Blechschmidt remarks that the internal coproduct exists and is left-exact for families of modules as well.
Our work in Section~\ref{ssec:coproducts-a-la-harting} simultaneously translates and generalises these results by constructing type-indexed colimits in arbitrary abelian categories in HoTT.
We then recover the analogue of Harting's result: when the indexing type is a set, the colimit specialises to the coproduct and is left-exact.
In general, however, the colimit fails to be left-exact (Example~\ref{exa:not-lex}).

The original construction of the internal coproduct of abelian groups was carried out in the internal language of an elementary topos.
This internal language was not well-developed at the time, and the paper~\cite{Har82}---which is entirely dedicated to this construction---weighs in at over 60 pages. 
In contrast, by working in HoTT our generalised construction goes through in just over 2 pages (\cref{ssec:coproducts-a-la-harting}).

The usual proof that AB5 implies AB4 replaces a discrete indexing category \( X \) (for a coproduct) by a filtered category (the finite subsets of \( X \)) and uses the fact that moving from one to the other does not change the colimit of a diagram.
However, in a constructive setting neither the Bishop-finite nor the (ordered) finite subsets of \( X \) form filtered categories unless \( X \) is decidable.
Harting's insight was to work with the category \( HX \) of (ordered) finite sub-multisets of \( X \) instead.
In Section~\ref{ssec:coproducts-a-la-harting} we define \( HX \) as a precategory associated to a $1$-type \(X\) in HoTT, then we show that \( HX \) is always sifted, and moreover filtered if \( X \) is a set.
In \cref{sec:sifted-and-filtered-precategories}, we develop the necessary theory of sifted and filtered colimits.

In \cref{sec:semantics} we interpret of our most central results into a higher topos \( \XX \) with a ring object \( R \).
Specifically, we show that the internal category of \(R\)-modules resulting from interpretation represents the presheaf sending an \( X \in \XX \) to the category of \((X{\times}R)\)-modules over \(X\)~(\cref{thm:R-mod-cat-represents}).
We repackage internal categories as {\em Rezk \((1,1)\)-objects}~(\cref{dfn:rezk-object}), which are \(2\)-restricted versions of \(0\)-truncated complete Segal objects.
Rezk \((1,1)\)-objects are easily seen to represent presheaves of categories, which is their main utility for us.

We also interpret type-indexed colimits of modules, which specialises to coproducts when the indexing type is a set.
For an object \( X \in \XX \), we get an adjunction
\( \colim{X} : \Mod{(X{\times}R)} \adj \Mod{R} : X \times (-) \)
where the left adjoint preserves products (\cref{thm:semantic-coprod-lex}).
If \( X \) is a set, then the left adjoint is left-exact.
To deduce (external) left-exactness from internal left-exactness (resulting from interpretation) we use an internalisation procedure~(\cref{dfn:internalisation}) that applies more generally, and may be of independent interest.

\subsection{Conventions}

We use the conventions and notation of~\cite{hottbook}.
Our terminology for category theory mirrors that of~\cite[Chapter~9]{hottbook} and~\cite{AKS15}, in particular we leave the ``univalent'' implicit when saying category (except in this paper's title).
When we consider abelian categories we do assume these are univalent, unlike the convention in~\cite{unimath}.
If \( \cat{D} \) and \( \cat{C} \) are precategories, we denote the functor precategory using exponential notation: \( \cat{C}^{\cat{D}} \).
For a functor \( F : \cat{C}^\cat{D} \) and a morphism \( \delta : d \to d' \) in \( \cat{D} \), we write \( F_\delta : F(d) \to F(d') \) for the morphism in \( \cat{C} \) obtained by applying \( F \).
If moreover \( \eta : G \Ra G' \) is a natural transformation of functors with domain \( \cat{C} \), then we will write \( \eta_F \) for the restriction of \( \eta \) along \( F \).

When we say something is a ``property of X'', we mean it in the formal sense of being a proposition.

\cref{sec:semantics} has its own section on notation.

\subsection{Acknowledgements}

I am grateful to both Raffael Stenzel and Nima Rasekh for helpful discussions about universes and representability.
Most of all, I am grateful to my advisor Dan Christensen for countless suggestions which have helped improve this text.

\section{Sifted and filtered precategories}
\label{sec:sifted-and-filtered-precategories}
We define sifted and filtered precategories, then prove that sifted (resp.\ filtered) colimits of sets commute with finite products (resp.\ finite limits).
In fact, we prove the stronger fact that filtered colimits commute with {\em finitely generated} limits (Definition~\ref{dfn:fg-precategory}).
This generalization lets us, for example, compute the fixed points of a filtered colimit of $G$-sets as the filtered colimit of the fixed points, for a finitely generated group $G$ (Corollary~\ref{cor:fixpoints-filtered}).

These are classical results in category theory, and the usual proofs go through in our context with some added care, which is what we supply.
The work builds on Chapters~9 and~10 of the HoTT Book~\cite{hottbook}.

Before we begin, we would like to emphasise that developing $1$-category theory in HoTT is unproblematic, as opposed to \( \infty \)-category theory.
We do now know how, or whether it is even possible, to represent current approaches to the latter in HoTT.
Nevertheless, we may speak about \( \infty \)-groupoids and functors between them, namely: an \( \infty \)-groupoid is simply a type, and a functor is simply a function.
In particular, if \( X \) is a type and \( \cat{C} \) is a category, then a function \( X \to \cat{C} \) is a functor from this point of view, and there is an obvious category \( \cat{C}^X \).

\subsection{Limits and colimits of sets}
We start by defining limits and colimits indexed by precategories.
When the codomain is a category, we show that the (co)limit of a functor is invariant under replacing the domain with its Rezk completion~(Lemma~\ref{lem:rezk-limit-invariant}).
For limits and colimits of sets, we show that the classical descriptions remain valid in our setting (Proposition~\ref{pro:set-limits}).
Lastly, when the indexing category is a groupoid (i.e.\ a $1$-type; see~\cite[Example~9.1.16]{hottbook}), we show that the limit and colimit are given respectively by the \( \Pi \)- and \( \Sigma \)-type of the underlying family.

\begin{dfn} \label{dfn:limit-hott}
  Let \( D : \cat{D} \to \cat{C} \) be a functor between precategories.
  A {\bf limit of \( D \)} is an object \( \lim{\cat{D}} D \) of \( \cat{C} \) representing the functor
  \( \cat{C}^{\cat{D}} (\const{\cat{D}}(-), D) : \cat{C}^{op} \to \hSet \).
  Dually, a {\bf colimit of \( D \)} is an object \( \colim{\cat{D}} D \) of \( \cat{C} \) representing the functor \( \cat{C}^{\cat{D}}(D, \const{\cat{D}}(-)) \).
\end{dfn}

As \( \hSet \) is a category, Theorem~9.5.9 in~\cite{hottbook} implies that the type of (co)limits of a functor \( D \) is a mere proposition.
Thus if a (co)limit exists, it is unique.

\begin{rmk}
  Consider a functor \( D : \cat{D} \to \cat{C} \).
  The data of a limit of \( D \) consists of an object \( \lim{\cat{D}} D : \cat{C} \) along with a natural isomorphism \( \delta : \cat{C}^\cat{D}(\const{\cat{D}}(-), D) \simeq \cat{C}(-, \lim{\cat{D}} D) \) witnessing representability.
  When we say that an object \( c : \cat{C} \) ``is the limit of \( D \)'', we mean that such a representability witness is specified.
  Of course, by the Yoneda lemma, such a witness consists exactly of an element in \( \cat{C}^\cat{D}(\const{\cat{D}}(c), D) \) defining a universal cone on \( D \).
  The dual story applies to colimits. 
\end{rmk}

Given a functor \( D : \cat{D} \to \cat{C} \) from a precategory to a category, we may factor \( D \) uniquely via the Rezk completion \( \rezk{D} \) as follows (see \cite[Chapter~9.9]{hottbook} for details):
\[ \begin{tikzcd}[row sep=tiny]
    \cat{D} \ar[dd, "\eta_{\cat{D}}" swap] \ar[dr, "D", bend left] \\
    & \cat{C} \\
    \rezk{D} \ar[ur, "\widehat{D}" swap, bend right]
  \end{tikzcd} \]
In particular, we have a natural comparsion map
\( \lim{\rezk{D}} \widehat{D} \longrightarrow \lim{\cat{D}} D \)
induced from \( \eta_{\cat{D}} \) be precomposition, and dually for the colimit.
The following lemma implies that that these comparison maps are isomorphisms, meaning we can freely move between the (co)limit of \( D \) and \( \widehat{D} \).

\begin{lem} \label{lem:rezk-limit-invariant}
  Let \( D : \cat{D} \to \cat{C} \) be a functor from a precategory to a category.
  The restriction maps
  \[ \eta_{\cat{D}}^* : \cat{C}^{\rezk{D}}(\const{\rezk{D}}(c),\widehat{D} ) \longrightarrow \cat{C}^{\cat{D}}(\const{\cat{D}}(c), D) \quad
    \t{and} \quad
    \eta_{\cat{D}}^* : \cat{C}^{\rezk{D}}(\widehat{D}, \const{\rezk{D}}(c)) \longrightarrow  \cat{C}^{\cat{D}}(D, \const{\cat{D}}(c)) \]
  are bijections natural in \( c : \cat{C} \).
  Consequently, the (co)limits of \( D \) and \( \widehat{D} \) coincide, if either exists.
\end{lem}

\begin{proof}
  The functor \( \eta_{\cat{D}} : \cat{D} \to \rezk{D} \) is a weak equivalence~\cite[Theorem~9.9.5]{hottbook}, thus mapping into \( \cat{C} \) induces an isomorphism \( \eta_{\cat{D}}^* : \cat{C}^{\rezk{D}} \to \cat{C}^{\cat{D}} \) by~\cite[Theorem~9.9.4]{hottbook}.
  Clearly, for every \( c : \cat{C} \), we have that \( \const{\rezk{D}}(c) \circ \eta_{\cat{D}} = \const{\cat{D}}(c) \) and \( \widehat{D} \circ \eta_{\cat{D}} = D \) by definition.
  The maps in question are actions of \( \eta_{\cat{D}}^* : \cat{C}^{\rezk{D}} \to \cat{D}^{\cat{D}} \) on specific hom-sets, which are (natural) bijections by full faithfullness.
\end{proof}

The usual descriptions of limits and colimits of sets are valid in HoTT.

\begin{pro} \label{pro:set-limits}
  Let \( \cat{D} \) be a small category, and \( D : \cat{D} \to \hSet \) a functor.
  \begin{enumerate}
  \item The limit of \( D \) exists, and is given by the set
    \[ \lim{\cat{D}} D = \{ x : \Pi_{\cat{D}} D \mid \Pi_{d,d' : \cat{D}} \Pi_{\delta : d \to d'} D_\delta (x_d) = x_{d'} \} \]
  equipped with the natural projections \( (\lim{\cat{D}} D \to D(d))_{d : \cat{D}} \) forming a universal cone.
  \item The colimit of \( D \) also exists, and is given by the set-quotient of \( \Sigma_{\cat{D}} D \) by the relation
    \[ (d, x) \sim (d', x') \defeq \ptr{ \Sigma_{\delta : d \to d'} D_\delta (x) = x' } \]
    equipped with the natural quotient maps \( (D(d) \to \Sigma_{\cat{D}} D / {\sim}  )_{d : \cat{D}} \) forming a universal cone.
    \end{enumerate}
\end{pro}

\begin{proof}
  The description of the limit $(1)$ results from computing \( \lim{\cat{D}} D \) via products and equalizers:
  \[ \begin{tikzcd}
      \lim{\cat{D}} D \rar[dashed]
      & \Pi_{d, d' : \cat{D}} \Pi_{\delta : d \to d'} D(d) \rar[shift left] \rar[shift right]
      & \Pi_{d : \cat{D}} D(d)
    \end{tikzcd} \]
  From the explicit descriptions of products and equalizers in \( \hSet \), we conclude. 
  Dually, the description of colimits $(2)$ is obtained by writing \( \colim{D} D \) via coproducts and coequalizers and using their respective descriptions as \( \Sigma \)-types and quotients in \( \hSet \).
\end{proof}

For indexing categories which are groupoids, both limits and colimits have a simpler description.

\begin{lem} \label{lem:groupoid-indexed-functor}
  Let \( \cat{G} \) be a groupoid, and \( \cat{C} \) a category.
  The forgetful map
  \( U : \cat{C}^{\cat{G}} \to (\cat{G} \to \cat{C}) \)
  which forgets functoriality is an equivalence.
  The inverse \( V \) sends a map \( f : \cat{G} \to \cat{C} \) to the functor \( V(f) \) acting as \( f \) on objects, and which sends a path \( \gamma : g =_{\cat{G}} g' \) to \( \idtoiso{G}(\ap{f}(\gamma)) \).
\end{lem}

\begin{proof}
  First of all the reader should convince themselves that the proposed definition of the inverse \( V \)  indeed constructs a functor \( V(f) \) from a general map of types \( f : \cat{G} \to \cat{C} \).
  It is then clear that \( V \) is a section of the forgetful map \( U \), so it remains to show that any functor \( F : \cat{G} \to \cat{C} \) is equal to the functor induced from its map on the underlying types.

  Clearly forgetting functoriality of \( F \) and then inducing functoriality produces the same map on the underlying types, by definition.
  Consider a general map \( \idtoiso{G}(\gamma) : g \to g' \) in \( \cat{G} \), where \( \gamma : g =_{\cat{G}} g' \).
  This is general since \( \cat{G} \) is a groupoid.
  We need to show that 
  \( F_{\idtoiso{\cat{G}}(\gamma)} = \idtoiso{G}(\ap{F}(\gamma)) \) as morphisms \( F(g) \to F(g') \) in \( \cat{C} \).
  But this follows by path induction on \( \gamma \).
\end{proof}

The lemma tells us that for functors from a groupoid into a category, we can choose to simply work with the underlying map of types.
Similar in spirit to Lemma~\ref{lem:rezk-limit-invariant}, the following proposition says that a (co)limit of sets is invariant under this change of perspective.

\begin{pro}
  Suppose \( \cat{G} \) is a groupoid, and let \( D : \cat{G} \to \hSet \) be a functor.
  The natural maps
  \( \lim{\cat{G}} D \to \Pi_{\cat{G}} D \)
  and \( \tr{\Sigma_{\cat{G}} D}_0 \to \colim{\cat{D}} D \) are bijections.
\end{pro}

\begin{proof}
  First we consider the limit.
  A family \( d : \Pi_{\cat{G}} D \) lies in the limit if and only if the proposition
  \[ \Pi_{g, g' : \Gc} \Pi_{f : \cat{G}(g,g')} D_f(d_g) = d_{g'} \]
  holds.
  Since \( \cat{G} \) is a groupoid, we can identify \( \cat{G}(g,g') \) with \( g =_{\cat{G}} g' \).
  The above then immediately follows by path induction, meaning the predicate defining the limit is a tautology.

  Similarly, we will show that the equivalence relation defining the colimit is trivial so that the set-quotient on \( \Sigma_{\cat{G}} D \) is simply given by set-truncation.
  Suppose \( (g_0,d_0) \sim (g_1, d_1) \) for the colimit relation defined in Proposition~\ref{pro:set-limits}.
  By definition there merely exists some \( f : g_0 \to g_1 \) such that \( D_f(d_0) = d_1 \).
  We wish to deduce that \( (g_0, d_0) = (g_1, d_1) \).
  Since this is a proposition, we may assume \( f \) actually exists.
  As before, we identify \( f \) with a path \( g_0 =_{\cat{G}} g_1\), using that \( \cat{G} \) is a groupoid.
  Then the existence of the path \( D_f(d_0) = d_1 \) implies exactly that \( (g_0, d_0) = (g_1, d_1) \), by characterisation of paths in \( \Sigma \)-types.
  In conclusion, the colimit relation \( \sim \) is just equality, hence the set-quotient \( \colim{\cat{G}} D \) is simply \( \tr{\Sigma_{\cat{G}} D}_0 \).
\end{proof}

\subsection{Sifted colimits}
\label{ssec:sifted}

We define sifted precategories and prove that sifted colimits commute with finite products in \( \hSet \).
To us, the main interest is that it lets us compute a sifted colimit of groups on the underlying sets, since any functor which commutes with products preserves group objects.

\begin{dfn}
  Let \( \cat{C} \) be precategory.
  \begin{enumerate}
    \item Let $c$ and  $c'$ be objects of \( \cat{C}\) and let \( n : \Nb \).
      A {\bf zig-zag from \( c \) to \( c' \) of length $n$} is a path \( c =_{\cat{C}} c' \) if $n \jeq 0$, or a sequence
      \( c \ra c_1 \la \cdots \ra c_{2n-1} \la c' \)
      of morphisms in \( \cat{C} \) if \( n \geq 1 \);
    \item The precategory \( \cat{C} \) is {\bf connected} if it is non-empty (i.e.\ \( \ptr{\cat{C}} \) holds) and for every two objects in \( \cat{C} \) there merely exists a zig-zag connecting them;
    \item Let \( \cat{C}' \) be a precategory.
      A functor \( F : \cat{C}' \to \cat{C} \) between precategories is {\bf final} if for every \( c : \cat{C} \), the slice precategory \( c / F \) is connected.
  \end{enumerate}
\end{dfn}

Being connected is a property of a precategory, and consequently being final is a property of a functor.
Our definition of zig-zags is tailored to facilitate formalization.
Restricting along a final functor leaves the colimit unchanged:

\begin{pro} \label{pro:final-functor-preserves-colimit}
  Let \( F : \cat{C}' \to \cat{C} \) and \( G : \cat{C} \to \cat{D} \) be functors between precategories.
  If \( F \) is final, then restriction along \( F \) is a natural bijection between functors \( \cat{D} \to \hSet \) as follows:
  \[ F^* : \cat{D}^\cat{C}(G, \const{\cat{C}}(d)) \longrightarrow \cat{D}^{\cat{C}'} (GF, \const{\cat{C}'}(d)) \]
  naturally in \( d : \cat{D} \).
  Consequently, the colimit of \( G \) coincides with the colimit of \( GF \), if either exists.
\end{pro}

\begin{proof}
  Let \( d : \cat{D} \).
  First of all, it is straightforward to check that \( F^* \) defines a natural transformation as stated.
  To prove that it is a natural isomorphism, we show that each component is a bijection.

  {\em Injectivity:} Suppose \( \eta, \eta' : G \Ra \const{\cat{C}}(d) \) are such that \( \eta_F = \eta'_F \).
  We want to show that for all \( c : \cat{C} \), \( \eta_c = \eta'_c \), which is a proposition.
  Let \( c : \cat{C} \), and pick a morphism \( f : c \to F(c') \) using that \( c / F \) is non-empty and the fact that we're proving a proposition.
  But then, by naturality of \( \eta \) and \( \eta' \), we have
  \[ \eta_c = \eta_{F(c')} \circ G_f = \eta'_{F(c')} \circ G_f = \eta'_c \]
  where the middle equation comes from \( \eta_F = \eta'_F \).
  Hence \( F^* \) is injective.

  {\em Surjectivity:} Consider a natural transformation \( \nu : GF \Ra \const{\cat{C}'}(d) \).
  For \( c : \cat{C}\), define the function
  \[ \phi(f) \defeq \nu_{c'} \circ G_f :  c / F \longrightarrow (G_c \to d) \]
  where \( f : c \to F(c') \).
  For \( f, f' : c / F \), one can easily show (using naturality of \( \nu \)) that \( \phi(f) = \phi(f') \) by induction over the length of a zig-zag from \( f \) to \( f' \).
  Consequently, \( \im{\phi} \) is a proposition and we may therefore factor \( \phi \) via its propositional truncation, producing
  \( \ptrel{\phi} : \ptr{c / F} \to \im{\phi} \to (G_c \to d) \).
  Thus we get a map \( g : G(c) \to d \) using the fact that \( c / F \) is non-empty.
  Doing this for all \( c : \cat{C} \) gets us a transformation \( \eta : \Pi_{c : \cat{C}} G(c) \to d \) which, by construction, satisfies \( \eta_F = \nu \).
  
  It remains to prove that \( \eta \) is natural. 
  Let \( g : c_0 \to c_1 \) be a morphism in \( \cat{C} \).
  We need to show that \( \eta_{c_0} = \eta_{c_1} \circ G_g \), which is a proposition.
  By finality of \( F \), we may choose \( f_0 : c_0 \to F(c'_0) \) and \( f_1 : c_1 \to F(c'_1) \) to obtain the following diagram:
  \[ \begin{tikzcd}
      G(c_0) \ar[dd, "G_g"] \rar["G_{f_0}"] & GF(c_0') \ar[dr, "\nu_{c'_0}", bend left] \ar[d, dotted] & {} \\
      & \vdots \rar[dotted] & d \\
      G(c_1) \rar["G_{f_1}"] & GF(c_1') \ar[ur, "\nu_{c'_1}" swap, bend right] \ar[u, dotted] & 
    \end{tikzcd} \]
  where the outer diagram is the one we wish to show commutes.
  Since \( c_0 / F \) is connected, the two maps \( f_0 \) and \( f_1 \circ g \) are connected by a zig-zag which, after applying \( G \), produces the dotted lines above.
  The left square then commutes by definition of a zig-zag, and the triangles on the right commute by naturality of \( \nu \).
  Inducting over the length of the zig-zag, we conclude that \( \eta \) is natural, as desired.
\end{proof}

\begin{dfn}
A precategory \( \cat{S} \) is {\bf sifted} if it is non-empty and \( \diag{S} : \cat{S} \to \cat{S} \times \cat{S} \) is final.
\end{dfn}

There are various equivalent classical definitions of siftedness.
We chose the one above to make the connection with final functors immediate, and to facilitate the proof of the following:

\begin{lem}
  If a precategory \( \cat{C} \) is non-empty and has binary coproducts, then \( \cat{C} \) is sifted.
\end{lem}

\begin{proof}
  Suppose \( \cat{C} \) is non-empty and has binary coproducts.
  Then for every \( (c_0, c_1) : \cat{C}^2 \), the slice precategory \( (c_0, c_1) / \diag{C} \) has an initial object given by the coproduct.
  Then we are done, since any category with initial object is connected (by zig-zags of length at most $1$).
\end{proof}

\begin{pro}
Sifted colimits of sets commute with finite products.
\end{pro}

\begin{proof}
  Let \( \cat{S} \) be a sifted precategory.
  The claim that colimits over \( \cat{S} \) commute with empty products follows from \( \cat{S} \) being non-empty.
  Consider two functors \( G, H : \cat{S} \to \hSet \), then we have the following natural bijections:
  \begin{align*}
    \colim{s : \cat{S}} \big ( G_s \times H_s \big ) &\simeq \colim{(s,t): S \times S} G_s \times H_t &\t{(Prop. \ref{pro:final-functor-preserves-colimit} applied to \( \diag{S} \))} \\
                                                       &\simeq \colim{s : \cat{S}} \colim{t : \cat{S}} G_s \times H_t \\
                                                       &\simeq \colim{s : \cat{S}} \big ( G_s \times \colim{t : \cat{S}} H_t \big ) &\t{(\( G_s \times - \) is cocontinuous)} \\
    &\simeq \colim{s : \cat{S}} G_s \times \colim{t : \cat{S}} H_t &\t{( \( - \times \colim{t : \cat{S}} H_t \) is cocontinuous)}
  \end{align*}
  where the second step can be checked directly.
  The product bifunctor \( \times \) preserves colimits in each variable, being a left adjoint.
\end{proof}


Sifted colimits of groups can be computed on the underlying sets.
Let \( U : \Gp \to \hSet \) be the forgetful functor in the following statement:

\begin{cor} \label{cor:sifted-preserve-groups}
  Let \( G : \cat{S} \to \Gp \) be a sifted diagram of groups.
  The set \( \colim{s : \cat{S}} U(G_s) \) carries a natural group structure which recovers \( \colim{\cat{S}} G \).
\end{cor}

\subsection{Filtered colimits} \label{ssec:filtered}
Filtered colimits of sets have particularly nice descriptions, and it is well known that they commute with finite limits, classically.
Less known is that fact that filtered colimits actually commute with {\em finitely generated limits} (Definition~\ref{dfn:fg-precategory}).
We start with the relevant definitions in our context.

\begin{dfn}
  A precategory \( \cat{F} \) is {\bf filtered} if the following propositions hold:
  \begin{enumerate}
  \item \( \cat{F} \) is non-empty;
  \item for any two objects \( c, c' : \cat{F} \), there merely exists an {\bf upper bound} \( c \to c'' \leftarrow c' \);
    \item for any two arrows \( f, g : c \to c' \), there merely exists some \( h : c' \to c'' \) such that \( hf = hg \).
  \end{enumerate}
\end{dfn}

By definition, filteredness is a property of a precategory.
We observe the following:

\begin{lem}
  Filtered precategories are sifted.
\end{lem}

It is straightforward to prove, by induction, that any finite family of objects in a filtered category merely admits an upper bound.
Similarly, any finite number of parallel arrows merely admit a (not necessarily universal) coequalizing arrow.
The more general fact is that filtered categories admit cone for finitely generated diagrams.

\begin{dfn} \label{dfn:fg-precategory}
  A precategory \( \cat{D} \) is {\bf finitely generated} if the underlying type of objects is Bishop-finite, and there exists a family of morphisms \( \Phi : \Pi_{i:I} \cat{D}(s_i, t_i) \) in \( \cat{D} \) indexed by a Bishop-finite set \( I \), such that every morphism in \( \cat{D} \) merely factors as follows:
  \[ \Pi_{m,m' : \cat{D}} \Pi_{g : m \to m'} \ptr{ \Sigma_{n : \Nb} \Sigma_{j : \fin{n} \to I} g = \Phi_{j(n-1)} \cdots \Phi_{j(0)} } \]
  where \( \fin{n} \) denotes the standard \(n\)-element set.
\end{dfn}

Observe that a finitely generated precategory is automatically a strict category.

\begin{pro} \label{pro:weakly-finitely-cocomplete}
  Let \( \cat{F} \) and \( \cat{D} \) be filtered and finitely generated categories, respectively.
  Any functor \( D : \cat{D} \to \cat{F} \) merely admits a cone.
\end{pro}

\begin{proof}
  Lemma~2.13.2 of~\cite{Bor94} readily generalizes to the case when \( \cat{D} \) is finitely generated.
\end{proof}

\begin{thm} \label{thm:filtered-colims-commute}
  Let \( \cat{F} \) and \( \cat{D} \) be filtered and finitely generated categories, respectively, and consider a functor \( D : \cat{F} \times \cat{D} \to \hSet \).
  The natural map \( \colim{\cat{F}} \lim{\cat{D}} D \to \lim{\cat{D}} \colim{\cat{F}} D \) is a bijection.
\end{thm}

\begin{proof}
  For finite categories \( \cat{D} \) the proof of~\cite[Theorem~2.13.4]{Bor94} goes through by careful use of finite choice.
  The generalisation to the when \( \cat{D} \) is finitely generated only requires straightforward modifications using \cref{pro:weakly-finitely-cocomplete} in the last part of Borceux' argument. 
\end{proof}

\begin{rmk}
  That filtered colimits commute with finite limits has been formalized in Mathlib~\cite{mathlib}.
  However, as opposed to HoTT, mathlib is based on a {\em classical} (as opposed to {\em constructive}) type theory assuming the law of the excluded middle and the axiom of choice.
\end{rmk}

As an application of our development thus far we have the following.
A group \( G \) is finitely generated if there exist a Bishop-finite generating set.
Recall that a $G$-set $X$ is simply a map \(X : BG \to \hSet \), and the fixed points of $X$ are given by \( \Pi_{BG} X \).

\begin{cor} \label{cor:fixpoints-filtered}
  Let \( G \) be a finitely generated group, and let \( X : \cat{F} \to (BG \to \hSet) \) be a filtered diagram of $G$-sets.
  The fixed points of the colimit is the colimit of the fixed points:
  \[ \Pi_{BG} \colim{\cat{F}} X \simeq \colim{x : \cat{F}} \Pi_{BG} X(x) \]
\end{cor}

\begin{proof}
  The category \( BG \) is the Rezk completion of the strict category \( B'G \) which has a single object with \( G \) as its endomorphisms.
  If \( G \) is a finitely generated group, then \( B'G \) is a finitely generated category in the sense of Definition~\ref{dfn:fg-precategory}.
  By \cref{lem:groupoid-indexed-functor} we have that \( \Pi_{BG}(-) = \lim{BG} (-) \), and by \cref{lem:rezk-limit-invariant} we can change the limits to be over \( B'G \).
  We conclude by the previous theorem, since \( B'G \) is finitely generated and \( \cat{F} \) is filtered.
\end{proof}

\section{The internal AB axioms}
\label{sec:ab}
The goal of this section is to show that for a ring \( R \) in HoTT, the category of $R$-modules satisfies the axioms AB3 through AB5 and has a generator---meaning it is a Grothendieck category (\cref{dfn:ab-n}).
Formally, a Grothendieck category is only assumed to satisfy AB3 and AB5, but we show that AB4 follows from AB5 (\cref{thm:coproduct-lex}).
It is straightforward to check that \( \Mod{R} \) is an abelian category in HoTT, and indeed this has already been formalised for \( R \jeq \Zb \) in the UniMath library~\cite{unimath}.
Moreover, \( R \) being a generator is simply a restatement of function extensionality.
What remains is to show that \( \Mod{R} \) satisfies the axioms AB3 through AB5.

We wish to treat families \( A : X \to \cat{A} \) in an abelian category \( \cat{A} \) indexed by an arbitrary type \( X \).
As pointed out at the beginning of Section~\ref{sec:sifted-and-filtered-precategories}, these are functors from an \( \infty \)-groupoid into a category.
Since \( \cat{A} \) is a category, its underlying type is $1$-truncated, and so we may factor any such family \( A \) through the $1$-truncation of $X$.
One checks that the $1$-truncation map \( \trel{-}_1 : X \to \tr{X}_1 \) induces an equivalence of categories by precomposition:
\[ \trel{-}_1^* : \cat{A}^{\tr{X}_1} \to \cat{A}^X \]
It follows, by an argument similar to the one in Lemma~\ref{lem:rezk-limit-invariant}, that the limit (resp. colimit) of a functor\footnote{Limits and colimits of functors from an \( \infty \)-groupoid into a category are defined in the obvious way.} \( A : X \to \cat{A} \) coincides with the limit (resp. colimit) of the $1$-truncation \( \trel{A}_1 : \tr{X}_1 \to \cat{A} \).
Thus when we discuss limits and colimits of such a family \( A \), we may assume that \( X \) is a $1$-type without loss of generality.

\subsection{Grothendieck categories}
\label{ssec:grothendieck-categories}

We define Grothendieck abelian categories in homotopy type theory, assuming the reader is familiar with additive and abelian precategories (whose definition can be found in~\cite{unimath}).
Be aware that by abelian {\em category} we do mean that it is a (univalent) category.
While much of our discussion likely works for abelian precategories as well, we are particularly interested in discussing families of objects,  which is most naturally done for categories.

\begin{dfn} \label{dfn:coproducts}
  Let \( \cat{A} \) be an additive category, and \( X \) a set.
  For a family \( A : X \to \cat{A} \), 
  the {\bf coproduct of \( A \)} (if it exists) is the colimit of \( A \), denoted \( \bigoplus_{x:X} A(x) \).
  Dually, the {\bf product of \( A \)} (if it exists) is the limit of \( A \), denoted \( \Pi_{x:X} A(x) \).
  If no confusion will arise, we often leave the variable \( x : X\) implicit.
\end{dfn}

Suppose \( \cat{A} \) is an additive category.
Then, by definition, finite products and coproducts in \( \cat{A} \) coincide, and we call these biproducts.
The word {\em finite} here means ``finitely iterated,'' i.e.\ pairwise biproducts carried out a finite number of times.

If \( X \) is a decidable set, and \( A : X \to \cat{A} \) is a family, then there is always a comparison map \( m : \bigoplus_X A \to \Pi_X A \) which is a monomorphism.
This is straightforward to prove in HoTT, and has been proved for families of modules in an elementary topos (with \( \Nb \)) by Tavakoli~\cite{Tav85}.
Of course, if \( X \) is the standard \(n\)-element set \( \fin{n} \) for some \( n : \Nb \), then the map $m$ is an isomorphism.
We deduce the following, since $m$ being an isomorphism is a proposition:

\begin{lem}
  Let \( \cat{A} \) be an additive category, and \( X \) a Bishop-finite type.
  For any family \( A : X \to \cat{A} \), the natural map \( m : \bigoplus_X A \to \Pi_X A \) is an isomorphism.
\end{lem}

It may come as a surprise that no such monomorphism $m$ need exist in general.
In fact, Harting demonstrates in~\cite[Remark~2.1]{Har82} that there might be no non-trivial map like $m$.
Her example is in the Sierpinski $1$-topos, but can be translated to the Sierpinski $\infty$-topos, which is a model of HoTT.
It follows that constructing a non-zero map \( \bigoplus_X A \to \Pi_X A \) for a general set \( X \) is impossible in HoTT.
Harting's example also demonstrates that the construction of arbitrary coproducts is tricky; for example, one cannot carve out \( \bigoplus_X A \) from \( \Pi_X A\) as those families with ``finite support.''

\begin{dfn} \label{dfn:ab-n}
  For an abelian category \( \cat{A} \) we have the following axioms.
  \begin{itemize}
  \item[(AB3)] for any small set \( X \) and family \( A : X \to \cat{A} \), the coproduct \( \bigoplus_X A \) exists in \( \cat{A} \);
  \end{itemize}

  Assuming \( \cat{A} \) satisfies AB3, we may moreover ask for:
  \begin{itemize}
  \item[(AB4)] for any small set \( X \), and any two families \( A, B : X \to \cat{A} \) along with a family of monomorphisms \( \eta : \Pi_{x : X} A(x) \to B(x) \), the induced map \( \bigoplus_X \eta : \bigoplus_X A \to \bigoplus_X B \) is a monomorphism;
  \end{itemize}

  If \( \cat{A} \) satisfies AB3, then it is automatically cocomplete\footnote{A general colimit can be computed via coproducts and coequalizers.} and we may ask for:
  \begin{itemize}
  \item[(AB5)] for any small filtered precategory \( \cat{F} \) and diagram \( F : \cat{F} \to \cat{A} \), the functor \( \colim{\cat{F}} : \cat{A}^\cat{F} \to \cat{A} \) preserves finite limits.
  \end{itemize} 
  A {\bf generator} of \( \cat{A} \) is an object \( G : \cat{A} \) such that for any two morphisms \( f, f' : A \to B \), we have
  \[ (\Pi_{g : G \to A} f g = f' g) \to  (f = f') \]
  The abelian category \( \cat{A} \) is {\bf Grothendieck} if it satisfies the axioms AB3 and AB5, and has a specified generator.
\end{dfn}

The axiom AB5 implies that the colimit functor \( \colim{\cat{F}} \) is exact for filtered precategories \( \cat{F} \).
In the next section we show that \( \Mod{R} \) is Grothendieck for any ring \( R \).

\subsection{Colimits of \texorpdfstring{$R$}{R}-modules}
\label{ssec:colimits-R-modules}

We consider a $1$-type \( X \) as a category and construct an adjunction:
\[ \colim{X} : \Mod{R}^X \adj \Mod{R} : \const{X} \]
When \( R \jeq \Zb \) and \( X \) is pointed and connected, \( \Mod{R}^X \) is the category of \( \pi_1(X) \)-modules.
We will see that the functor \( \colim{X} \) computes the {\em coinvariants} of a \( \pi_1(X) \)-module.
Dually, the functor \( \lim{X} \) computes the {\em invariants} (but \( \lim{X} \) needs no discussion: it is simply the limit of the underlying sets).
   
As in classical algebra, the forgetful functor \( U : \Mod{R} \to \Ab \) reflects limits and colimits.
Thus by constructing \( \colim{X} \) for families of abelian groups, we extend it to families of modules via \( U \).

\begin{pro} \label{pro:colimit-ab}
  Let \( X \) be a small $1$-type.
  We have an adjunction
  \( \colim{X} : \Ab^X \adj \Ab : \const{X} \).
\end{pro}

\begin{proof}
  We start by constructing the functor \( \colim{X} \).
  Let \( A : X \to \Ab \).
  Via Theorem~4 of~\cite{BvDR}, we may instead consider the corresponding family $\K{A}{2}$ of pointed, $1$-connected $2$-types.
  The colimit of \( \K{A}{2} \) among all types is then \( \Sigma_{X} \K{A}{2} \) by \cref{lem:groupoid-indexed-functor}, whereas the colimit among pointed types is the pushout 
  \[ \begin{tikzcd}
      X \rar["\pt{}"] \dar \ar[dr, "\ulcorner" at end, phantom] & \Sigma_{X} \K{A}{2} \dar[dashed] \\
      1 \rar[dashed] & \bigvee_{X} \K{A}{2}
      \end{tikzcd} \]
    called the {\em indexed wedge}.
    Thus the colimit of $\K{A}{2}$ among pointed $2$-types is $\tr{ \bigvee_{X} \K{A}{2} }_2$ by~\cite[Section~7.4]{hottbook}.
    Moreover, by Theorem~7.3.9 in~\cite{hottbook} we have that 
    \[ \tr{\Sigma_{X} \K{A}{2}}_1 \simeq \tr{\Sigma_{x:X} \tr{\K{A(x)}{2}}_1 }_1 \simeq X \]
    using that \( \K{A(x)}{2} \) is $1$-connected for all \( x : X \).
    From this we deduce that $1$-truncating the pushout square above produces
  \[ \begin{tikzcd}
      X \rar["\id{}"] \dar \ar[dr, "\ulcorner" at end, phantom] & X \dar[dashed] \\
      1 \rar[dashed] & 1
      \end{tikzcd} \]
    since pushouts commute with truncation.
    In particular, \( \tr{ \bigvee_{X} \K{A}{2} }_1 \) is $1$-connected.
    Finally, since \( \tr{ \bigvee_{X} \K{A}{2} }_2 \) has the desired universal property among pointed $2$-types, it certainly has it among pointed, $1$-connected $2$-types, being one itself.
    Now we apply \( \pi_2 \), the inverse to \( \K{-}{2} \), to define our functor on objects:
    \[ \colim{X}(A) \defeq \pi_2 \left ( \bigvee_{X} \K{A}{2} \right ) \]

    As defined, \( \colim{X} \) is a composite of the functors, hence is itself a functor.
\end{proof}

\begin{cor} \label{cor:modR-bicomplete}
  Let \( R \) be a ring.
  The category \( \Mod{R} \) is complete and cocomplete.
\end{cor}

\begin{proof}
  We reduce to \( R \jeq \Zb \) since the forgetful functor reflects both limits and colimits.
  For limits, note that \( \Ab \) has small products given simply by the \( \Pi \)-type associated to a family \( X \to \Ab \) indexed by a set.
  Since \( \Ab \) has equalizers, it is complete.
  Dually, Proposition~\ref{pro:colimit-ab} produces small coproducts by letting \(X\) be a set.
  Since \( \Ab \) has coequalizers, it is cocomplete.
\end{proof}

\begin{thm} \label{thm:mod-R-grothendieck}
  The category \( \Mod{R} \) is Grothendieck.
\end{thm}

\begin{proof}
  That \( R \) is a generator is an immediate consequence of function extensionality.
  By the previous corollary, \( \Mod{R} \) is cocomplete and therefore satisfies AB3.
  The axiom AB5 follows from Theorem~\ref{thm:filtered-colims-commute}, since filtered colimits of $R$-modules may be computed on the underlying sets.
\end{proof}

At this point, it is not obvious that \( \Mod{R} \) satisfies AB4.
This is shown in the next section (\cref{thm:coproduct-lex}).
In the remaining part of this section we discuss \( \colim{X} \) when \( X \) is the classifying space of a group.

\begin{dfn}
  Let \( G \) be a group.
  A family \( A : BG \to \Ab \) is a {\bf \( G \)-module}.
  The {\bf invariants} of $A$ comprise the abelian group \( A_G \defeq \lim{BG} A \), and the {\bf coinvariants} comprise the abelian group \( A^G \defeq \colim{BG} A \).
\end{dfn}

Using the fact that limits of abelian groups may be computed on the underlying sets, along with the concrete description of limits in Proposition~\ref{pro:set-limits}, we see that 
\( A_G = \{ a : A \mid \Pi_{g: G} ga = a \} \),
which is the usual definition of the invariants.
Writing \( \colim{BG} A \) as a coequalizer produces
\[ \begin{tikzcd}
    \bigoplus_{g : G} A \rar[shift left, "a \mapsto ga"] \rar[shift right, "a \mapsto a" swap] & A \rar[dashed] & A^G
  \end{tikzcd} \]
from which we see that \( A^G \) is the quotient of \( A \) by the subgroup \( \langle a - ga \mid g : G \rangle \), which is the usual definition of the coinvariants.

\begin{rmk}
After Definition~\ref{dfn:coproducts}, we discussed Harting's counterexample to the existence of a monomorphism \( \bigoplus_X A \to \Pi_X A \).
The obstruction is set-theoretic, namely decidability of \( X \).
Of course, this also means there is in general no monomorphism \( \colim{X} A \to \lim{X} A \), but it is much easier to produce a counterexample to this.
For example, if we consider the \( \Zb / 2 \)-module \( \Zb \) given by the negation action, then the coinvariants are \( \Zb^G = \Zb / 2 \) but the invariants are \( \Zb_G = 0 \).
\end{rmk}

In~\cite{Har82}, Harting carried out her specific construction of the internal coproduct of abelian groups so as to prove that the resulting coproduct functor was left-exact (in particular, it preserves monomorphisms).
For us, the internal coproduct is \( \colim{X} \) for a set \( X \).
In the next section we generalise (the analogue of) Harting's result by proving that any abelian category which satisfies AB3 and AB5 has a left-exact coproduct functor.
Before doing so, we demonstrate that \( \colim{X} \) generally fails to be left-exact when \( X \) is not a set.

\begin{exa} \label{exa:not-lex}
  Let \( G \defeq \Zb / 2 \).
  A \( G \)-module is then an abelian group equipped with an automorphism which squares to the identity.
  Consider the \( G \)-module \( \Zb \) equipped with the negation automorphism \( n \mapsto -n \), and the \( G \)-module \( \Zb \times \Zb \) equipped with the ``swap'' automorphism \( (a,b) \mapsto (b,a) \).
  We have a \( G \)-equivariant monomorphism \( 1 \mapsto (-1, 1) : \Zb \to \Zb \times \Zb \) which fails to induce a monomorphism on the coinvariants.
  Explicitly, the respective coinvariants are \( \Zb^G = \Zb / 2 \) and \( (\Zb \times \Zb)^G = \Zb \).
  Of course, there are of course no non-trivial maps \( \Zb/2 \to \Zb \), and certainly no monomorphisms.
  Consequently the functor \( \colim{B (\Zb / 2)} \), which computes the coinvariants, is not left-exact.
\end{exa}

\subsection{AB5 implies AB4}
\label{ssec:coproducts-a-la-harting}
We prove that AB4 follows from AB5 for any abelian category, as is familiar in ordinary homological algebra.
The classical proof goes by replacing a discrete indexing category \( X \) (for a coproduct) by a filtered category (the finite subsets of \( X \)) sharing the same colimit, then applying AB5.
However, in a constructive setting neither the category of Bishop-finite subsets of \(X\), nor the category of ordered finite subsets of \(X\), form filtered categories unless \(X\) is decidable.
For this reason we will work with ordered finite sub-{\em multisets}, i.e.\ general maps of the form \( \fin{n} \to X \) as opposed to only the injections.

We wish to point out that this is how Harting constructs the {\em internal coproduct} of abelian groups in an elementary topos (with \( \Nb \)) in~\cite{Har82}, though she does not phrase things in terms of the AB axioms.
While the goal of this construction is to realise a coproduct as a filtered colimit, we find it interesting to observe that Harting's ``set-theoretic'' description readily generalises to untruncated indexing types, as well as abelian categories \( \cat{A} \).
Specifically, given a family \( A : X \to \cat{A} \) indexed by an arbitrary type $X$, we replace \( A \) by a sifted diagram \( GA : HX \to \cat{A} \) sharing the same colimit (if it exists).
If \( X \) is a set, so the colimit is the coproduct, then \( \bigoplus_{x:X} A(x) \) will be a filtered colimit, as desired.

Our first objective is to define the precategory \( HX \) of finite sub-multisets of any $1$-type \( X \).
In general \( HX \) will be sifted, and even filtered when \( X \) is a set.
The latter situation is essentially the one studied in~\cite[177--178]{JW78}.
Throughout this section, let \( X \) be a $1$-type (unless otherwise stated), and let \( \cat{A} \) be an abelian category.
We implicitly identify \( X^n \) and \( \fin{n} \to X \) where \( \fin{n} \) is the standard \(n\)-element set.

\begin{dfn} \label{dfn:HX}
We make the $1$-type \( \Sigma_{n: \Nb} X^n \) into a precategory \( HX \) by letting the morphisms be commuting triangles (with specified witness of commutativity):
\[ (n,x) \to (m,y) \defeq  \Sigma_{f : \fin{n} \to \fin{m}} x =_{X^n} y \circ f \]
for \( (n,x) , (m,y) : \Sigma_{n : \Nb} X^n \).
Since $X$ is a $1$-type, so is \( X^n \).
Therefore the hom-type defined above is a set, as required for being a precategory.
Checking that this indeed defines a precategory is straightforward.
\end{dfn}

We observe the following lemma, which in general fails for the precategories of Bishop-finite or finite ordered subsets of \( X \).

\begin{lem} \label{lem:HX-sifted}
  The precategory \( HX \) has coproducts. In particular, it is sifted.
\end{lem}

The next proposition is~\cite[Lemma~4.4]{JW78} translated to our setting.\footnote{The precise relation being that a presheaf is flat if and only if its category of elements (``total category'') is filtered; see also~\cite[Propositio n~1.3]{JW78}.}
Note that \( \Sigma_{n : \Nb} X^n \) is a set if $X$ is, and $HX$ is then a strict category.
In this situation, when discussing morphisms in \( HX \) we may omit references to the commutativity witnesses.

\begin{pro} \label{pro:HX-filtered}
  If \( X \) is a set, then \( HX \) is filtered.
\end{pro}

\begin{proof}
  The proof of~\cite[Lemma~4.4]{JW78} can be carried out almost word-for-word in our setting.
  The perhaps only non-obvious step requires that the coequalizer of two parallel arrows \( f,g : \fin{n} \to \fin{m} \) itself be finite, i.e.\ of the form \( \fin{l} \).
  This holds because the relation induced by \( f \) and \( g \) on \( \fin{m} \) is decidable, and the quotient of a finite set by a decidable relation is also finite.
\end{proof}

Now we show how to replace diagrams \( X \to \cat{A} \) by diagrams \( HX \to \cat{A} \).

\begin{cst}
  We construct a functor \( G : \cat{A}^X \to \cat{A}^{HX} \) as follows.
  For a family \( A : X \to \cat{A} \), let \( GA(n,x) \defeq \bigoplus_{i : \fin{n}} A(x_i) \).
  For a morphism \( (f, p) : (n, x)  \to (m, y) \) in \( HX \), we have the path \( p : \Pi_{i : \fin{n}} x_i = y_{f(i)} \) which induces a morphism \( A_p : \bigoplus_{i : \fin{n}} A(x_i) \to \bigoplus_{i : \fin{n}} A(y_{f(i)}) \) by transporting and functoriality of biproducts.
  We define the morphism \( GA_{(f, p)} : \bigoplus_{i : \fin{n}} A(x_i) \to \bigoplus_{j : \fin{m}} A(y_j) \) as the composite:
  \[ \textstyle \bigoplus_{i : \fin{n}} A(x_i) \xlongrightarrow{A_p} \bigoplus_{i : \fin{n}} A(y_{f(i)}) \xlongrightarrow{\nabla} \bigoplus_{j : \fin{m}} A(y_j) \]
  where the last map sums over the fibres of \( f \).
  The sum is well-defined since it is finite: any function between finite types has decidable fibres, and a decidable subset of a finite type is finite, hence \( \fib{f}(j) \) is finite for all \( j : \fin{m} \).

  Checking that \( GA \) defines a functor is straightforward.
  Lastly, the obvious functor \( h_X : X \to HX \) defined by \( h_X(x) \defeq (1,x) \) makes the following diagram commute:
  \[ \begin{tikzcd}[row sep=small]
      X \ar[dd, "h_X" swap] \ar[dr, "A", bend left=20] \\
      & \cat{A} \\
      HX \ar[ur, "GA" swap, bend right=20]
    \end{tikzcd} \]
\end{cst}

The following is the analogue of~\cite[Proposition~2.5]{Har82} in our setting.

\begin{lem} \label{lem:H-lex}
  The functor \( G : \cat{A}^X \to \cat{A}^{HX} \) respects limits.
\end{lem}

\begin{proof}
  Let \( A : \cat{D} \to \cat{A}^X \) be a diagram whose limit exists.
  For all \( (n, x) : HX \), we have
  \[ \textstyle G(\lim{d:\cat{D}} A_d)(n,x) \jeq \bigoplus_{j   : \fin{n}} \lim{d: \cat{D}} A_d(x_j)
    = \lim{d : \cat{D}} \bigoplus_{j : \fin{n}} A_d(x_j) \jeq \lim{d : \cat{D}} GA_d(n,x) \]
  using that limits in functor categories are computed pointwise, and that \( \bigoplus_{\fin{n}} \) preserves limits.
\end{proof}

Before the next proposition, we require a lemma:

\begin{lem} \label{lem:Ap-path-ind}
  Let \( n : \Nb \), and \( A : X \to \cat{A} \).
  Consider an object \( M : \cat{A} \) along with a family \( \eta : \Pi_{x:X} A(x) \to M \).
  For any path \( p : x = x' \) in \( X^n \), the following diagram commutes:
  \[ \begin{tikzcd}[row sep=small]
      \bigoplus_{i : \fin{n}} A(x_i) \ar[dr, "\bigoplus_{i} \eta_{x_i}", bend left=20] \ar[dd, "A_p" swap] \\
      & M \\
      \bigoplus_{i : \fin{n}} A(x'_i)  \ar[ur, "\bigoplus_{i} \eta_{x'_i}" swap, bend right=20]
    \end{tikzcd} \]
\end{lem}

\begin{proof}
  By path induction on \( p \).
\end{proof}

Now we prove that passing between \( A \) and \( GA \) leaves the colimit unchanged (if it exists).

\begin{pro} \label{pro:GA-colimit}
  Let \( A : X \to \cat{A} \).
  Restriction along the functor \( h_X : X \to HX \) is an isomorphism
  \[ h_X^* : \cat{A}^{HX}(GA, \const{HX}(M)) \to \cat{A}^X(A, \const{X}(M)) \]
  natural in \( M : \cat{A} \).
  Consequently, the colimits of \( A \) and \( GA \) coincide, when they exist.
\end{pro}

\begin{proof}
  We construct an explicit inverse \( e \) to \( h_X^* \).
  Let \( M : \cat{A} \), and let \( \eta : A \Ra \const{X}(M) \) be a natural transformation, i.e.\ a family \( \eta : \Pi_{x:X} A(x) \to M \).
  Given such a family \( \eta \), we extend it to a natural transformation \( e(\eta) : GA \Ra \const{HX}(M) \) using the biproduct, as follows.
  For \( (n, x) : HX \), let 
  \[ \textstyle e(\eta)_{(n,x)} \defeq \bigoplus_{i} \eta_{x_i} : \bigoplus_{i : \fin{n}} A(x_i) \to M \]
  Thus we have defined a transformation \( e(\eta) \), and now we check naturality.

  Let \( (f, p) : (n, x) \to (m, y) \) be a morphism in \( HX \).
  Our task is to verify that outer triangle in the following diagram commutes:
  \[ \begin{tikzcd}
  \bigoplus_{i : \fin{n}} A(x_i) \ar[dr, "\bigoplus_i \eta_{x_i}" swap, bend right=20] \rar["A_p"]
  & \bigoplus_{i : \fin{n}} A(y_{f(i)}) \rar["\nabla"] \dar[dashed]
  & \bigoplus_{j : \fin{m}} A(y_j) \ar[dl, "\bigoplus_j \eta_{x_j}", bend left=20]  \\
  & M
  \end{tikzcd} \]
  where the dashed line is \( \bigoplus_{i : \fin{n}} \eta_{y_{f(i)}} \).
  The inner-left triangle commutes by Lemma~\ref{lem:Ap-path-ind}.
  That the inner-right triangle commutes can be immediately checked on each component \( i : \fin{n} \).
  Thus we conclude that \( e(\eta) \) is a natural transformation.

  From the construction it is clear that \( h_X^* \circ e = \id\).
  For the other equality, let \( \nu : GA \Ra \const{HX}(M) \) be a natural transformation.
  Given some \( (n, x) : HX \), then for any \( i : \fin{n} \) we have the morphism in \( HX \) on the left, whose filler is the reflexivity path:
  \[ \begin{tikzcd}[row sep=small]
      \fin{1} \ar[dd, "i"] \ar[dr, "x_i", bend left=20] & && A(x_i) \ar[dr, "\nu_{(1,x_i)}", bend left=20] \ar[dd, "GA_i"] \\
      & X && & M \\
      \fin{n} \ar[ur, "x" swap, bend right=20] & && \bigoplus_{j : \fin{n}} A(x_j) \ar[ur, "\nu_{(n,x)}" swap, bend right=20]
    \end{tikzcd} \]
  The vertical arrow in the right triangle is the inclusion, which is also given by functoriality of \( GA \).
  The right triangle commutes by naturality of \( \nu \).
  By the universal property of the $n$-fold biproduct, we have that \( \nu_{(n,x)} = \bigoplus_{i : \fin{n}} \nu_{(1,x_i)} \).
  This means that \( \nu = e(h_X^*(\nu)) \), and consequently \( \id = e \circ h_X^* \).
\end{proof}

The proposition tells us that the following diagram commutes, whenever \( \cat{A} \) is cocomplete:
\[ \begin{tikzcd}
    & \cat{A}^{HX} \ar[dr, "{\colim{HX}}"] \\
    \cat{A}^X \ar[ur, "G"] \ar[rr, "\colim{X}"]&& \cat{A}
\end{tikzcd} \]
From this we deduce the following results.

\begin{cor} \label{cor:colim-X-preserves-finite-products}
  The functor \( \colim{X} : \Mod{R}^X \to \Mod{R} \) preserves finite products.
\end{cor}

\begin{proof}
  We know \( G \) preserves limits, and \( \colim{HX} \) preserves products since $HX$ is sifted.
\end{proof}

\begin{thm} \label{thm:coproduct-lex}
  Suppose \( \cat{A} \) is an abelian category satisfying AB3 and AB5.
  For any set \( X \), the functor \( \bigoplus_X : \cat{A}^X \to \cat{A} \) is left-exact.
  In particular, \( \cat{A} \) satisfies AB4.
\end{thm}

\begin{proof}
  The assumption that \( \cat{A} \) satisfies AB5 means that the functor \( \colim{HX} \) is exact, because \( HX \) is filtered when \( X \) is a set by Proposition~\ref{pro:HX-filtered}.
  Since \( G \) respects limits, we conclude from the diagram above that \( \colim{X} \) (i.e.\ \( \bigoplus_X \)) is left-exact.
\end{proof}

\section{Semantics}
\label{sec:semantics}
We interpret the most central results from the previous sections into an \( \infty \)-topos \( \XX \), as made possible by recent developments on the semantics of Homotopy Type Theory~\cite{KL21, LS20, Shu19, Boe20}.
Specifically, we work out the interpretation of categories of modules (\cref{thm:R-mod-cat-represents}) and colimits of modules indexed by a type (\cref{thm:semantic-coprod-lex}).

Thus far we have studied categories of abelian groups and modules, as well as abstract abelian categories in HoTT.
Semantically, these yield structures in our chosen \( \infty \)-topos \( \XX \).
For example, we will see that the ``internal category'' \( \iAb \) obtained by interpretation represents---in the sense of \cref{dfn:represents}---the presheaf
\[ X \longmapsto \Ab(\XX \kslice X) : \XX^\op \longrightarrow \Cat \]
which sends an object \( X \in \XX \) to the ordinary category of (relatively \(\kappa\)-compact) abelian groups over \( X \).

Before setting off our assumptions need some care.
Any Grothendieck \( \infty \)-topos \( \cat{X} \) can be presented by a type-theoretic model topos \( \cat{M} \) according to~\cite{Shu19}.
Assuming an inaccessible cardinal \( \kappa \), the latter admits a univalent universe \( \Typep^\kappa \to \Type^\kappa \) for relatively \( \kappa \)-presentable fibrations~\cite[Definition~4.7]{Shu19} supporting the interpretation of HoTT.\footnote{Modulo certain classes of higher inductive types, which we do not use.}
Moreover, Stenzel~\cite{Ste19} proves that the universe presents a classifying object~\cite[Section~6.1.6]{HTT} for relatively \( \kappa \)-compact morphisms\footnote{The difference in terminology ($\kappa$-presentable vs.\ $\kappa$-compact) is unfortunate. As we work in the \( \infty \)-setting, we will employ Lurie's terminology, i.e.\ ``\(\kappa\)-compact''~\cite[Definition~6.1.6.4]{HTT}, when necessary.} in \( \XX \).

We will require a small fragment of the theory of complete Segal objects~\cite{Ras18} in \( \XX \) (also called {\em internal \(\infty \)-categories}~\cite{Mar21} or {\em Rezk objects}~\cite{RV22}).
As our model of the (large) \( \infty \)-category \( \ooCat \) of \( \infty \)-categories, we choose the \( \infty \)-category of complete Segal spaces.
Though our arguments will clearly be model-independent, certain specific constructions require a choice, and this is a convenient one for our purposes.

\begin{ntn*}
  We will write \( \XXk \) for the sub-\(\infty\)-category of \( \kappa \)-compact objects in \( \XX \), and for an object \( X \in \XX \) we form the slice \( (\XX \kslice X) \) of relatively \( \kappa \)-compact morphisms into \( X \).
  The \(1\)-topos of \(0\)-truncated objects in \( \XX \) is \( \sets{\XX} \).
  The functor \( \core{-} : \ooCat \to \spaces \) picks out the \( \infty \)-groupoid core of an \( \infty \)-category, and \( \spaces \) is the \( \infty \)-category of spaces (also called \(\infty\)-groupoids).
  For complete Segal spaces \( \core{-} \) simply picks out the zeroth space.
The universal map in \( \XX \) presented by Shulman's univalent universe will be written \( \Typep \to \Type \), leaving \( \kappa \) implicit.
No confusion will arise as no other universes will be around.
Notions in \( \XX \) resulting from interpretation will be denoted in typewriter font.
For example we will be considering the universe \( \iSet \) classifying $\kappa$-compact $0$-truncated objects.
In particular, we leave the \( \kappa \) implicit in the notation of the universe of sets (or abelian groups, or $R$-modules).
\end{ntn*}

\subsection{Rezk \texorpdfstring{\((1,1)\)}{(1,1)}-objects}
The first goal of this section is to repackage the internal categories in \( \XX \) obtained by interpretation into structures which conveniently represent presheaves of \(1\)-categories.
We begin by explaining how \(1\)-categories can be associated to an \( \infty \)-category such as \( \XX \).
An ordinary category \( \Cc \) is incarnated as a simplicial space through its {\bf classifying diagram} \( \cd(\Cc) \)~\cite[Section~3.5]{Rez01}:
\begin{equation} \label{eqn:classifying-diagram}
  \begin{tikzcd}
\cd(\Cc) :=  \bigg ( \cdots \rar["\vdots", "\vdots" swap] & \core{\Cc^{[2]}} \rar[shift left=.6em] \rar \rar[shift right=.6em]
    & \lar[shift left=.3em] \lar[shift right=.3em] \core{\Cc^{[1]}} \rar[shift left=.3em] \rar[shift right=.3em]
    & \lar \Cc^\simeq \bigg )
  \end{tikzcd}
\end{equation}
where we used \( \core{-} \) to denote the Kan complex obtained from the groupoid core of a $1$-category, and \( [n] \) denotes the usual poset with \( n+1 \) elements.
This classifying diagram is a complete Segal space, and there is a Quillen adjunction 
\( \fc : \ooCat \adj \Cat : \cd \) which exhibits \( \Cat \) as precisely the $1$-truncated complete Segal spaces~\cite[Theorem~5.11]{CL20}.
The left adjoint \( \fc \) is the fundamental category functor.
By identifying \( \Cat \) with its image under the embedding \( \cd \), we may speak about presheaves of $1$-categories on $\XX$.
On the \(1\)-categorical level, \( \fc \) factors through the category \( \sSet_\rD \) of \(2\)-restricted simplicial spaces.
We therefore expect the corresponding \(\infty\)-functor to factor through \( \spaces_\rD \), though we do not provide a proof.

The following are the structures into which we will repackage internal categories.

\begin{dfn} \label{dfn:rezk-object}
  A {\bf Segal \( (1,1) \)-object in \( \XX \)} is a \(2\)-restricted simplicial object \( \Cc : \rD \to \XX \) satisfying the three following conditions:
  \begin{enumerate}[wide=2em, leftmargin=*]
  \item[(truncation)] the {\bf structure map} \( (\dom, \cod) : \Cc_1 \to \Cc_0 \times \Cc_0 \) is $0$-truncated in \( \XX \);
  \item[(Segal condition)] the natural map \( \Cc_2 \to \Cc_1 \times_{\Cc_0} \Cc_1 \) is an equivalence;
  \item[(associativity)] the following two composites agree:
    \[ \begin{tikzcd}
        \Cc_1 \times_{\Cc_0} \Cc_1 \times_{\Cc_0} \Cc_1 \ar[rr, shift left=.3em, "\id \times \circ"] \ar[rr, shift right=.3em, "\circ \times \id" swap]
        && \Cc_1 \times_{\Cc_0} \Cc_1 \rar["\circ"] & \Cc_1
      \end{tikzcd} \]
    where \( \circ : \Cc_1 \times_{\Cc_0} \Cc_1 \xra{\sim} \Cc_2 \xra{\delta^2_1} \Cc_1 \).
  \end{enumerate}
  If moreover the square below below is a pullback, then \( \Cc \) is a {\bf Rezk \((1,1)\)-object}:
  \begin{equation} \label{eqn:rezk-condition}
    \begin{tikzcd}
    \Cc_0 \rar["{(\id, \id, \id)}"] \dar["\Delta"] & \Cc_1 \times_{\Cc_0} \Cc_1 \times_{\Cc_0} \Cc_1 \dar["{(f,g,h) \mapsto (fg, gh)}"] \\
    \Cc_0 \times \Cc_0 \ar[r, "{\id \times \id}"] & \Cc_1 \times \Cc_1
  \end{tikzcd}
\end{equation}
The Segal (or Rezk) \((1,1)\)-object \( \Cc \) is {\bf locally small} if the structure map is relatively \(\kappa\)-compact.
\end{dfn}

It is straightforward to interpret Definition~9.1.1 from~\cite{hottbook} to get the data of a precategory in \( \XX \).
We allow the underlying type of a precategory to be any object of \( \XX \), not necessarily classified by \( \Type \).
Our next lemma states how this data can be repackaged into a Segal \((1,1)\)-object.

\begin{lem} \label{lem:precategory-segal}
  Precategories in \( \XX \) correspond to locally small Segal \((1,1)\)-objects, and categories to locally small Rezk \((1,1)\)-objects.
\end{lem}

\begin{proof}
  Given a precategory \( \iC \), we define a \(2\)-restricted simplicial object \( \iC_\bullet \) as follows.
  Let \( \iC_0 \defeq \iC \), and write \( (\dom, \cod) : \iC_1 \to \iC_0 \times \iC_0 \) for the total space of the hom \( \iC(-,-) : \iC \times \iC \to \iSet \) with its projection.
  The identity maps \( \id : \Pi_{c : \iC} \iC(c,c) \) give a section \( \iC_0 \to \iC_1 \) of both \( \dom \) and \( \cod \).
  Now let \( \iC_2 \defeq \Sigma_{a, b, c : \iC_0} \Sigma_{f : \iC(a,b)} \Sigma_{g : \iC(b,c)} \Sigma_{h : \iC(a,c)} gf = h \) be the object of commuting triangles in \( \iC \).
  Then \( \iC_\bullet \) is a \(2\)-restricted simplicial object in \( \XX \) with face maps given by projections, and degeneracies induced by \( \id \).
  Clearly \( \iC_\bullet \) satisfies the truncation condition, and is associative.
  The map \( (f, g) \mapsto (f,g,gf, \refl{gf}) \) is easily shown to be an inverse to the natural map \( \iC_2 \to \iC_1 \times_{\iC_0} \iC_1 \) in HoTT, thus we conclude that \( \iC_\bullet \) is a Segal \((1,1)\)-object.
  It is locally small by construction.

  It is similarly straightforward to produce a precategory from a locally small Segal \((1,1)\)-object.
  Under this correspondence, univalence of a precategory is equivalent to the square (\ref{eqn:rezk-condition}) being a pullback, so we conclude that categories correspond to Rezk \((1,1)\)-objects.
\end{proof}

Now we explain in what sense Rezk \((1,1)\)-objects represent presheaves of ordinary categories. 
By our discussion of the functor \( \fc \) on the previous page, it is clear that to recover the fundamental category of a classifying diagram it suffices to recover the lower three simplicial levels.
This parallels the fact that categories in HoTT only yield {\em \(2\)-restricted} simplicial objects (as opposed to ``unrestricted'' ones) and leads us to the following notion of representability.

\begin{dfn} \label{dfn:represents}
  Let \( \Cc : \XX^\op \to \Cat \) be a presheaf of \( 1 \)-categories on \( \XX \).
  A Rezk \((1,1)\)-object \( \iC : \rD \to \XX \) {\bf represents} \( \Cc \) if there is a specified natural equivalence \( \eta : \XX(-,\iC_\bullet) \simeq i^*_2 \Cc \) of functors \( \XX^\op \to \spaces_\rD \), where \( i^*_2 \) is the restriction along the inclusion \( \rD \to \Delta \).
\end{dfn}

We will use this notion of representability when working out the semantics of the category of sets and categories of modules in the next sections.
The reader who is mainly interested in those representability results (e.g.\ \cref{thm:R-mod-cat-represents}) may skip ahead to the next section.
The remaining parts of this section are only needed for \cref{thm:semantic-coprod-lex}.

Any statement about (pre)categories in HoTT yields a statement about locally small (Segal) Rezk \((1,1)\)-objects by translating across the correspondence of \cref{lem:precategory-segal}.
For example, one can check that products of (pre)categories correspond to levelwise products of (Segal) Rezk \((1,1)\)-objects.
Our next statement is that functor precategories interpret to the internal hom of Segal \((1,1)\)-objects.

If \( F, G : \cat{C} \to \cat{D} \) are two functors between (pre)categories in HoTT, then we can represent natural transformations \( F \Ra G \) as functors \( \eta : \cat{C} \times [1] \to \cat{D} \) such that \( \eta \lvert_{\cat{C} \times \{0\}} = F \) and \( \eta \lvert_{\cat{C} \times \{1 \}} = G \).
The precategory \( [n] \) interprets to the Segal \((1,1)\)-object \( \rd^n \) which is the \(2\)-restriction of the obvious Segal object \( \Delta^n \).
We have the following:

\begin{lem}
  Let \( \Cc \) and \( \Dc \) be locally small Segal \((1,1)\)-objects in \( \XX \).
  \begin{enumerate}
  \item the object of functors \( \iFun(\Cc, \Dc) \) obtained by interpretation represents the presheaf
    \[ X \longmapsto (\XX / X)_\rD(X \times \Cc, X \times \Dc) : \XX^\op \longrightarrow \spaces \]
    where the base change functor \( X \times (-) \) is applied levelwise;
  \item the Segal \((1,1)\)-object \( \iFun(\Cc, \Dc)_\bullet \) obtained by interpreting the functor category is equivalent to
\[ \begin{tikzcd}
    \iFun(\Cc \times \rd^2, \Dc) \rar[shift left=.6em] \rar[shift right=.6em] \rar & 
    \lar[shift left=.3em] \lar[shift right=.3em] \iFun(\Cc \times \rd^1, \Dc) \rar[shift left=.3em] \rar[shift right=.3em] & 
    \lar \iFun(\Cc \times \rd^0, \Dc) 
  \end{tikzcd} \]
where the degeneracy and face maps are induced by the \( \rd^n \)'s.
If \( \Dc \) is Rezk, then so is \( \iFun(\Cc, \Dc)_\bullet \).
  \end{enumerate}
\end{lem}

\begin{proof}
  It is straightforward to see that functors between precategories in HoTT interpret to simplicial maps between the corresponding Segal \((1,1)\)-objects.
  Then (1) follows by stability of interpretation across base change.

  By representing natural transformations as functors, we see that \( \iFun(\Cc \times \rd^1, \Dc) \) is the total space of the map \( \iFun(\Cc, \Dc)^2 \to \iSet \) which sends two functors to the set of natural transformations between them.
  Thus we get the first and second levels of (2).
  Finally, the third level is naturally equivalent to \( \iFun(\rd^2, \iFun(\Cc, \Dc)_\bullet) \) in HoTT, and the latter is clearly equivalent to the space of commuting triangles in \( \iFun(\Cc, \Dc) \).
  These equivalences clearly assemble to a simplicial map, so we are done.
 \end{proof}

 We note that by combining part (1) and (2) of the lemma, we get a formula for the presheaf represented by \( \iFun(\Cc, \Dc)_\bullet \).

When working with Segal and Rezk \((1,1)\)-objects we may use category-theoretical language as long as the relevant interpretation has been worked out, or is clear from the context.
We also note that we can take \( \XX \) to be the \(\infty\)-topos \( \spaces \) of spaces, and in this case we will use the terminology {\bf Segal} and {\bf Rezk \((1,1)\)-spaces} for emphasis.

Our next proposition asserts that functor categories interpret to the internal hom in \( \XX_\rD \).
In order to prove this, we require a lemma.
Recall the terminal geometric morphism \( \XX(1, -) : \XX \adj \spaces : \ell \).
Applying this adjunction levelwise, we get an induced geometric morphism \( \XX(1,-) : \XX_\rD \adj \spaces_\rD : \ell_* \), and left-exactness of \( \ell_* \) implies that it preserves Segal and Rezk \((1,1)\)-objects.
Note that \( \ell_*(\rd^n) = \rd^n \) where we leave the ambient \(\infty\)-topos implicit.

\begin{lem} \label{lem:internal-simplicial-hom}
  Let \( \Cc, \Dc \in \XX_\rD \).
  For \( X \in \XX \) and \( n \in \{ 0, 1, 2 \} \), we have:
    \[ \begin{tikzcd}
      \XX(X, \Dc^\Cc)_n \simeq (\XX / X)_\rD(X \times \Cc \times \Delta^n_{\leq 2}, X \times \Dc)
    \end{tikzcd} \]
\end{lem}

\begin{proof}
  By stability of the internal hom across base change, we can assume \( X = 1 \).
  We then have:
    \[ \XX_\rD (\Cc \times \rd^n, \Dc) \simeq \XX_\rD(\rd^n, \Dc^\Cc) \simeq \spaces_\rD(\rd^n, \XX(1,\Dc^\Cc)) \simeq \XX(1,\Dc^\Cc)_n \]
    where the first equivalence is by cartesian-closedness of \( \XX_\rD \), the second equivalence comes from the adjunction \( \ell_* \dashv \XX(1,-) \), and the last equivalence is the Yoneda lemma.
\end{proof}

\begin{pro}
  Let \( \Cc \) and \( \Dc \) be locally small Segal \((1,1)\)-objects in \( \XX \).
  The \(2\)-restricted simplicial objects \( \Dc^{\Cc} \) and  \( \iFun(\Cc, \Dc)_\bullet \) in \( \XX \) are equivalent.
\end{pro}

\begin{proof}
By the Yoneda lemma, it suffices to show that the functors \( \XX(-, \Dc^\Cc) \) and \( \XX(-, \iFun(\Cc, \Dc))_\bullet \) of the form \( \XX^\op \to \spaces_\rD \) are equivalent.
But this is immediate by combining the previous two lemmas.
\end{proof}

It follows that the internal hom between Segal \((1,1)\)-objects is itself a Segal \((1,1)\)-object, and even Rezk if the codomain is.

  If \( \Cc \) is a Rezk \((1,1)\)-object in \( \XX \), then the internal limit of a functor \( F : \Dc \to \Cc \) in \( \XX_\rD \) defines a global point \( \ilim{\Dc} F \in \XX(1,\Cc_0) \), if the internal limit exists.
  Of course, so does the limit of an {\em external} functor \( G : D \to \XX(1,\Cc) \) in \( \spaces_{\rD} \).
  We now explain how such external functors \( D \to \XX(1,\Cc) \) can be {\em internalised} to functors in \(\XX_\rD\), and we prove that this procedure does not change the limit or colimit.

  \begin{dfn} \label{dfn:internalisation}
    Let \( \Cc \) be a Rezk \((1,1)\)-object in \( \XX \), and \( D \) a Rezk \((1,1)\)-space.
    Given a functor \( A : D \to \XX(1,\Cc) \), its {\bf internalisation} is the transpose \( \int{A} : \ell_*(D) \to \Cc \) across the adjunction \( \ell_* \dashv \XX(1,-) \).
  \end{dfn}

  To show that internalisation does not change the (co)limit, we require a lemma.
  The reader may find it interesting to compare it with \cite[Example~2.39]{Joh77}.

  \begin{lem}
    Let \( \Cc \) be a Rezk \((1,1)\)-object in \( \XX \), and \( D \) a Rezk \((1,1)\)-space.
    The Rezk \((1,1)\)-spaces \( \XX(1, \Cc^{\ell_*(D)}) \) and \( \XX(1,\Cc)^D \) are equivalent.
  \end{lem}

  \begin{proof}
    Using \cref{lem:internal-simplicial-hom} and the adjunction \( \ell_* \dashv \XX(1,-) \), for \( n \in \{0,1,2\} \) we have:
    \[ \XX(1, \Cc^{\ell_*(D)})_n \simeq \XX_\rD(\ell_*(D) \times \rd^n, \Cc) \simeq \spaces_\rD (D \times \rd^n, \XX(1,\Cc)) \simeq (\XX(1,\Cc)^D)_n \]
    where the second equivalence uses that \( \ell_* \) preserves products (being left exact), then transposes across the adjunction.
    The third equivalence is \cref{lem:internal-simplicial-hom} applied to Rezk \((1,1)\)-spaces.
    Using basic properties of adjunctions, one can check that these equivalences assemble to a simplicial map.
  \end{proof}

  The category of sets in HoTT interprets to a Rezk \((1,1)\)-object \( \iSet_\bullet \) which features in the next proposition, and is the main topic of study in the next section.
  For the following proof, we only use that \( \XX(1,\iSet_\bullet) \) has a terminal object and therefore a global sections functor \( \Gamma : \XX(1, \iSet_\bullet) \to \Set \).
  Observe that if \( \Cc \) is a Rezk \((1,1)\)-object in \( \XX \), then the Rezk \((1,1)\)-space \( \XX(1,\Cc) \) is ``enriched'' over \( \XX(1, \iSet_\bullet) \).
  A study of this ``enrichment'' is beyond the scope of this work, and our convention will be to implicitly apply \( \Gamma \) so that the hom \( \XX(1,\Cc)(-,-) \) lands in \( \Set \). 

  \begin{pro} \label{pro:internalisation}
    Let \( \Cc \) be a locally small Rezk \((1,1)\)-object in \( \XX \), and let \( A : D \to \XX(1,\Cc) \) be a functor between Rezk \((1,1)\)-spaces.
    If the internal limit \( \ilim{\ell_*(D)} \int{A} \) in \( \Cc \) exists, so does the limit of \(A\) and we have a canonical isomorphism \( \ilim{\ell_*(D)} \int{A} \simeq \lim{D} A \) in \( \XX(1,\Cc) \).
  \end{pro}

  \begin{proof}
    Suppose the limit of \( \int{A} \) in \( \Cc \) exists, giving an equivalence of functors \( \Cc^\op \to \iSet_\bullet \)
    \[ \Cc(-, \ilim{\ell_*(D)} \int{A}) \: \simeq \: \Cc^{\ell_*(D)}(\iconst{\ell_*(D)}(-), \int{A}) \]
    Applying \( \XX(1, -) \), we get an equivalence between certain functors \( \XX(1,\Cc)^\op \to \XX(1,\iSet_\bullet) \), and by further post-composing with the global sections map \( \Gamma : \XX(1,\iSet_\bullet) \to \Set \), we get an equivalence
    \begin{equation} \label{eqn:internalisation-1}
      \XX(1,\Cc)(-, \ilim{\ell_*(D)} \int{A}) \simeq \XX(1, \Cc^{\ell_*(D)})(\iconst{\ell_*(D)}(-), \int{A})
    \end{equation}
    between functors \( \XX(1,\Cc)^\op \to \Set \).
    We have an equivalence \( \XX(1, \Cc^{\ell_*(D)}) \simeq \XX(1,\Cc)^D \) by the previous lemma, which sends \( \iconst{\ell_*(D)} \) to \( \const{D} \) and \( \int{A} \) to \( A \).
    On hom-spaces, this means we have:
    \begin{equation} \label{eqn:internalisation-2}
      \XX(1, \Cc^{\ell_*(D)})(\iconst{\ell_*(D)}(-), \int{A}) \simeq \XX(1,\Cc)^D(\const{D}(-), A)
    \end{equation}
    Combining the equivalences (\ref{eqn:internalisation-1}) and (\ref{eqn:internalisation-2}), we see that \( \ilim{\ell_*(D)} \int{A} \) is the limit of \( A \), as desired.
  \end{proof}

 The proposition and its proof dualises to colimits, but we will only need it for limits.

\subsection{The universe of sets}
We show that the Rezk \((1,1)\)-object \( \iSet_\bullet \) produced by interpretation represents the presheaf \( \sets{\XX \kslice (-)} : \XX^\op \to \Cat \) in the sense of \cref{dfn:represents}.
First we show a lemma that proves useful for these kinds of representability results.

Recall that the universe \( \Type \) is an {\em object classifier}~\cite[Section~6.1.6]{HTT} and therefore represents (in the usual sense) the presheaf of spaces
\( \core{\XX \kslice (-)} : \XX^\op \to \spaces \).
We will be interested in types which classify certain structures in \( \XX \).
For example, given a ring \( R \in \sets{\XXk} \), we will see that there is a map \( \Rmodstr : \iSet \to \Type \) which classifies \(R\)-modules in \( \XXk \), meaning that the mapping space \( \XX(X, \Sigma_{A:\iSet} \Rmodstr(A)) \) is the groupoid of \(R\)-modules in \( (\XX \kslice X) \) (\cref{thm:R-mod-cat-represents}).
The following lemma gives a description of these mapping spaces for general type families.

\begin{lem} \label{lem:classify}
  Let \( P : Z \to \Type \) be a type family in \( \XX \), and \( X \in \XX \).
  The outer square in the following diagram is a pullback:
  \[ \begin{tikzcd}
      \XX(X,\Sigma_{z:Z} P(z)) \dar \rar & \XX(X,\Typep) \dar \rar["\sim"] & \coree{(\XX \kslice X)_*} \dar \rar & \spaces_* \dar \\
      \XX(X,Z) \rar["{\XX(X,P)}"] & \XX(X,\Type) \rar["\sim"] & \coree{(\XX \kslice X)} \rar["{\coree{\Gamma}}"] & \spaces
  \end{tikzcd} \]
where the functor \( \coree{\Gamma} \) is the restriction of the global points functor \( \Gamma : (\XX \kslice X) \to \spaces \) to the core, and \( (\XX \kslice X)_* \) is the \(\infty\)-category of pointed (relatively \(\kappa\)-compact) objects over \( X \).
\end{lem}

\begin{proof}
  The right square is manifestly a pullback, and so is the left square since \( \XX(X,-) \) preserves limits.
  Since \( \Typep \) classifies pointed objects, we get the middle square.
  By pullback pasting we conclude that the outer square is a pullback.
\end{proof}

Recall that \( \iSet \) is defined as the total space of the map \( \ishset : \Type \to \Type \) sending a type \( A \) to the proposition in \( \XX \) which holds (or has a global point) if and only if \( A \) is a \(0\)-truncated object.
The universe \( \iSet \) of sets classifies \(0\)-truncated objects:

\begin{lem} \label{lem:universe-set-represents}
  The object \( \iSet \) represents the presheaf of spaces \( \coree{\sets{\XX \kslice (-)}} \).
\end{lem}

\begin{proof}
  By applying the previous lemma to the type family \( \ishset : \Type \to \Type \), we see that \( \XX(1,\iSet) \) is the sub-\(\infty\)-groupoid of \( \XX(1,\Type) \) on those objects \( A : \Type \) for which \( \ishset(A) \) holds.
  Since \( \ishset(A) \) holds if and only if \( A \) is \(0\)-truncated, \( \XX(1,\iSet) \) is equivalent to the groupoid of \(0\)-truncated objects in \( \XXk \).

  For general \( X \), we always have that families \( X \to \iSet \) correspond to families \( X \to X \times \iSet \) over \( X \).
  Since pullback-stability of the universe implies that \( X \times \iSet \) is a universe of sets in \( \XX / X \), we reduce to the case \( X = 1 \) just treated by pulling back over \( X \).
\end{proof}

Applying~\cite[Theorem~4.4]{Ras21} to the the universal map \( p : \Typep \to \Type \) yields a complete Segal object \( N(p) \) in \( \XX \) which represents (in the usual sense) the presheaf
\( (\XX \kslice (-)) : \XX^\op \to \ooCat \).
The \(2\)-restriction of \( N(p) \) is equivalent to the following \(2\)-restricted simplicial object \( \Type_\bullet \) in \( \XX \):
\[ \begin{tikzcd}
    \sum_{X, Y, Z : \Type} Y^X \times Z^Y \rar[shift left=.6em] \rar \rar[shift right=.6em]
    & \lar[shift left=.3em] \lar[shift right=.3em] \sum_{X, Y : \Type} Y^X \rar[shift right=.3em] \rar[shift left=.3em]
    & \lar \Type
\end{tikzcd} \]
To be explicit, we know that the function types modelled by the universe interpret to the internal hom in \( \XX \), so the first simplicial level is simply the type-theoretic notation for Rasekh's description of \( N(p)_1 \), and the second level is given by the Segal condition.

For the Rezk \((1,1)\)-object \( \iSet_\bullet \), the object of morphisms is simply given by internal homs in \( \XX \):
\[ \iSet_1 := \Sigma_{X,Y : \iSet} (X \to Y) \to \iSet \times \iSet  \]

The following provides the semantics of the category of sets in HoTT.

\begin{pro} \label{pro:cat-set-represents}
  The Rezk \((1,1)\)-object \( \iSet_\bullet \) represents the presheaf \( \sets{\XX \kslice (-)} : \XX^\op \to \Cat \) in the sense of \cref{dfn:represents}.
\end{pro}

\begin{proof}
  Let \( X \in \XX \).
  We need to produce a natural equivalence \( \eta : \XX(X,\iSet_\bullet) \simeq i^*_2 (\sets{\XX \kslice X}) \) of \(2\)-restricted simplicial spaces.
  By \cref{lem:universe-set-represents}, we get a natural equivalence of the zeroth levels.
  Since the global points of the internal hom give the external hom, \cref{lem:classify} tells us that \( \XX(X,\iSet_1) \) is naturally equivalent to the groupoid of arrows in \( \sets{\XX \kslice X} \).
  These two equivalences clearly assemble to an equivalence of \(1\)-restricted simplicial objects, whereby we get an induced equivalence of the second simplicial levels via the Segal condition.
  The latter equivalence automatically respects the face maps \( \delta^2_0 \) and \( \delta^2_2 \) as well as the degeneracies.
  We need to check that it respects the composition map \( \delta^2_1 \).
  But this follows from the fact that function types interpret to the internal hom in \( \XX \).
\end{proof}

\subsection{The universe of \texorpdfstring{$R$}{R}-modules}
Let \( R \) be a ring object in \( \sets{\XX} \).
We show that the Rezk \((1,1)\)-object \( \iMod{R}_\bullet \) of \(R\)-modules in \( \XX \) represents the presheaf sending an object \( X \in \XX \) to the ordinary category of modules over the ring \( X \times R \in \sets{\XX \kslice X} \) (\cref{thm:R-mod-cat-represents}).

The key ingredient we used to prove that \( \iSet \) classifies \(0\)-truncated objects was that \( \ishset(A) \) has a global point if and only if \( A \) is a \(0\)-truncated object.
Similarly, to say what \( \iMod{R} \) classifies we need to understand the global points of \( \Rmodstr(A) \).

\begin{lem} \label{lem:internal-external-R-module-structures}
  Let \(R\) be a ring object in \( \sets{\XX} \).
  For all \( A \in \sets{\XX} \), global points of the object \(
  \tt{\(R\)-mod-str}(A) \)
  biject with \(R\)-module structures on the object \( A \) in \( \XX \).
\end{lem}

\begin{proof}
It is well-known that the global points of \( A \), \( A^A \), \(A^{A \times A} \), and \( A^{R \times A} \) biject respectively with the set of points of $A$, the set of endomorphisms of $A$, the set of binary operations on \( A \), and set of maps \( R \times A \to A \) in \( \Ec \).
 One can check the global points functor \( \Gamma \) sends the limit diagram carving out the subobject \( \tt{\(R\)-mod-str}(A) \) of internal \(R\)-module structures on \( A \)
to the limit diagram carving out the (external) set of $R$-module structures on \( A \) from inside the set
\[ \Gamma A \times \Gamma(A^A) \times \Gamma(A^{A \times A}) \times \Gamma(A^{R \times A}) \]

 Since \( \Gamma \) preserves limits, we are done.
\end{proof}

For a ring \( R \in \sets{\XXk} \) and an object \( X \in \XX \), recall that \( X {\times} R \in \sets{\XX \kslice X} \) is a ring over \( X \).
We now show that \( \iMod{R} \) classifies \(R\)-modules in \( \XXk \).

\begin{pro} \label{pro:R-mod-universe-represents}
  Let \( R \) be a ring in \( \sets{\XXk} \).
  The object \( \iMod{R} \) represents the space-valued presheaf
  \[ X \longmapsto \coree{\Mod{(X{\times}R)}} : \XX^\op \longrightarrow \spaces \]
\end{pro}

\begin{proof}
  First of all, by pullback-stability of the universe, we have that \( \iMod{(X{\times}R)} \simeq X \times \iMod{R} \) over \( X \).
  Since families \( X \to \iMod{R} \) correspond to families \( X \to X \times \iMod{R} \) over \( X \), we can assume \( X = 1 \) by pulling back over \( X \).
  
  As defined, \( \iMod{R} \) is the total space of \( \tt{\(R\)-mod-str} \).
  Combining Lemmas \ref{lem:internal-external-R-module-structures} and \ref{lem:classify}, we see that \( \XX(1,\iMod{R}) \) is naturally equivalent to the groupoid of \(R\)-modules in \( \XXk \), as desired.
\end{proof}

We recall how internal objects of homomorphisms in \( \XX \) are constructed.

\begin{dfn} \label{dfn:internal-hom}
  \begin{enumerate}
  \item Let \( A \) and \( B \) be abelian group objects in \( \sets{\XX} \).
    The {\bf object of group homomorphisms} \( \iiAb(A,B) \) is the following equaliser in \( \XX \):
    \[ \begin{tikzcd}[column sep=large]
        \iiAb (A,B) \rar[dashed]
        & B^A \ar[rr, shift left, "f \longmapsto f +_B f"] \ar[rr, shift right, "f \longmapsto f \circ  (+_A)" swap]
        && B^{A \times A}
      \end{tikzcd} \]
  \item Let $R$ be a ring in \( \sets{\XX} \), and let \( A \) and \( B \) be two $R$-modules.
    Write \( \alpha_X : R \times X \to X \) for the $R$-action on an $R$-module $X$.
    The {\bf object of $R$-module morphisms \( \iiMod{R}(A,B) \)} is the following equaliser in \( \XX \):
    \[ \begin{tikzcd}[column sep=large]
        \iiMod{R}(A,B) \rar[dashed]
        & \iiAb(A,B) \ar[rr, shift left, "f \longmapsto f \circ \alpha_A"] \ar[rr, shift right, "f \longmapsto \alpha_B(\id_R \times f)" swap]
        && B^{R \times A}
      \end{tikzcd} \]
  \end{enumerate}
\end{dfn}

It is not hard to see, using an argument similar to the proof of \cref{lem:internal-external-R-module-structures}, that the global points of \( \iiMod{R}(A,B) \) are actual \(R\)-module homomorphisms from \( A \) to \( B \).
Additionally, the object \( \iMod{R}(A,B) \) coming from interpretation is equivalent to \( \iiMod{R}(A,B) \), since it interprets to the same equaliser.

\begin{thm} \label{thm:R-mod-cat-represents}
  Let $R$ be a ring in \( \sets{\XXk} \).
  The Rezk \((1,1)\)-object \( \iMod{R}_\bullet \) represents the presheaf
  \[ X \longmapsto \Mod{(X{\times}R)} : \XX^\op \longrightarrow \Cat \]
  in the sense of~\cref{dfn:represents}.
\end{thm}

\begin{proof}
  Let \( X \in \XX \).
  By our definition of representability, we need to produce a natural equivalence \( \eta : \XX(X,\iMod{R}_\bullet) \simeq i^*_2 \Mod{(X{\times}R)} \) of \(2\)-restricted simplicial spaces, and \cref{pro:R-mod-universe-represents} gets us \( \eta_0 \).

  For the first level, recall that \( \iMod{R}_1 \) is the total space of the family \( \iMod{R}(-,-) : \iMod{R}^2 \to \Ab \).
  By our discussion just above, applying \( \XX(X,-) \) to this family recovers the internal hom of \((X{\times}R)\)-modules restricted to the groupoid core.
  Since the global points of the internal hom of modules recovers the external hom of modules, we conclude by \cref{lem:classify} that there is a natural equivalence \( \eta_1 :  \XX(X,\iMod{R}_1) \simeq \core{\Mod{(X{\times}R)}^{[1]}} \).
  By construction, this equivalence respects the two projection maps sending a homomorphism to its domain and codomain.
  We also need to check that it respects the degeneracy map \( \id : \XX(X,\iMod{R}) \to \XX(X,\iMod{R}_1) \) which picks out the identity.
  This follows from the corresponding fact for sets, since \( \id \) here is induced by the degeneracy \( \XX(X,\iSet) \to \XX(X, \iSet_1) \) and equality of \( R \)-module homomorphisms can be checked on the underlying maps.
  We conclude that \( \eta_0 \) and \( \eta_1 \) assemble to a map of \( 1 \)-restricted simplicial spaces.

  For the second level, we have a candidate for the equivalence
  \( \eta_2 : \XX(X,\iMod{R}_2) \to \core{\Mod{R}^{[2]}} \)
  given by \( \eta_1 \times_{\eta_0} \eta_1 \) and using the Segal condition and that \( \XX(X,-) \) preserves limits.
  By construction \( \eta_2 \) respects the two face maps \( \delta^2_0 \) and \( \delta^2_2 \), since these are just pullback projections.
  In addition, \( \eta_2 \) respects the two degeneracy maps since these are induced by \( \id \) above, and \( \eta_1 \) respects \( \id \).
Finally, we need to check that \( \eta_2 \) respects composition.
But composition of \( R\)-module homomorphisms is defined by composing the underlying maps, and since we can check equality of \( R \)-module homomorphisms on the underlying maps, this follows from the corresponding statement for sets. 

We conclude that \( \eta \) defines a natural equivalence of \(2\)-restricted simplicial objects, as desired.
\end{proof}

Finally, we explain the semantics of \cref{thm:coproduct-lex} and \cref{cor:colim-X-preserves-finite-products} for module categories.
To any object \( X \) in \( \XX \) (more generally, any morphism) we have the usual sequence of adjoints
\( \Sigma_X \dashv X \times (-) \dashv \Pi_X \).
The right adjoints automatically lift to categories of modules, being left-exact.
By the internal cocompleteness of categories of modules, we have a corresponding leftmost adjoint \( \colim{X} \).
By \cref{cor:colim-X-preserves-finite-products} \( \colim{X} \) preserves internal products, and \cref{thm:coproduct-lex} implies that it is internally left-exact whenever \( X \) is \(0\)-truncated.
On global points, we deduce the following:

\begin{thm} \label{thm:semantic-coprod-lex}
  Let $R$ be a ring object in \( \sets{\XXk} \), and let \( X \in \XX \).
  We have an adjunction:
  \[ \begin{tikzcd}
    \Mod{(X{\times}R)} \ar[rr, shift right=.5em, "\colim{X}" swap, "\top"] && \ar[ll, shift right=.5em, "{X\times(-)}" swap] \Mod{R}
  \end{tikzcd} \]
where \( \colim{X} \) preserves products.
If \( X \) is \(0\)-truncated, then \( \bigoplus_X \jeq \colim{X} \) is left-exact.
\end{thm}

We emphasise that this is an {\em external} statement about ordinary categories, and left-exactness refers to preservation of finite limits in the usual (external) sense.

\begin{proof}
  From Theorem~\ref{thm:coproduct-lex} we get an adjunction between Rezk \((1,1)\)-objects:
  \[ \textstyle \icolim{X} : \iMod{R}^X \adj \iMod{R} : X {\times} (-) \]
  which by the previous theorem yields the adjunction of our statement on global points.
  Explicitly, it is clear that the right adjoint corresponds to base change, so we conclude that the left adjoints must agree.

  By \cref{pro:internalisation}, the limit of a finite family \( A : n \to \Mod{(X{\times}R)} \) can be computed as the limit of the internalisation \( \int{A} : \ell_*(n) \to \iMod{R}^X \).
  The category \( \ell_*(n) \) is simply the interpretation of \( \fin{n} \) and is therefore internally finite.
  Hence \( \icolim{X} \) preserves the limit of \( \int{A} \) by \cref{cor:colim-X-preserves-finite-products}.
  We have:
  \[ \lim{i : n} \colim{X} A(i) \simeq \ilim{i : \ell_*(n)} \icolim{X} \int{A}(i) \simeq \icolim{X} \ilim{\ell_*(n)} \int{A} \simeq \colim{X} \lim{n} A \]
  where the first and third equivalences use \cref{pro:internalisation} for limits.
  
  If \( X \) is a set, to see that \( \bigoplus_X \) is left-exact it suffices to show that it preserves products and equalisers.
  We already know it preserves products.
  Applying \( \ell_* \) to an external equaliser diagram produces the internal one in \( \XX \) obtained by interpretation.
 The claim then follows by the same argument as above.
\end{proof}

We end by discussing the relation of this theorem to \cite[Theorem~2.7]{Har82}.

\begin{rmk}
Harting's construction of the left-exact coproduct applies in any elementary \(1\)-topos (with \( \Nb \)).
The \(1\)-topos \( \sets{\XXk} \) is---in particular---an elementary topos, hence Theorem~2.7 of loc.\ cit.\ implies the \(0\)-truncated and \( R \jeq \Zb \) case of our theorem above.
Conversely, a Grothendieck \(1\)-topos \( \Sh_0(\Cc) \) is equivalent to the \(0\)-truncated fragment of the \(\infty\)-topos \( \Sh_\infty (\Cc) \) of \(\infty\)-sheaves on the same site.
Consequently, we recover Harting's theorem for \( \Sh_0(\Cc) \) by applying our theorem to \( \Sh_\infty (\Cc) \).

It is not yet known whether any elementary \(1\)-topos can be realised as the \(0\)-truncated fragment of some {\em elementary} \(\infty\)-topos~\cite{Ras22}.
Nor is it known whether Homotopy Type Theory has semantics in the latter.
If these both hold, then our theorem would in turn imply (indeed, generalise) Harting's theorem in the elementary setting as well.

In \cite{Har82}, the construction of the internal coproduct of abelian groups occupies almost 60 pages, partly because the internal language of an elementary \(1\)-topos was not well-developed at the time.
However, once the construction was complete, left-exactness followed by general results of~\cite{Joh77}.
In contrast, our generalised construction is essentially contained in  \cref{ssec:coproducts-a-la-harting}, and weighs in at just over 2 pages.
The analogue of the general results of \cite{Joh77} in our setting---or at least the parts we needed---are embodied by \cref{pro:internalisation}, and various of our results in HoTT. 
\end{rmk}

\printbibliography

\end{document}